\documentclass[reqno,centertags,a4paper,11pt]{amsart}
\usepackage{amssymb}
\usepackage{amsthm}
\usepackage{eucal}
\usepackage{a4wide}
\usepackage{hyperref}

\newcommand{\R}{\mathbb R}

\numberwithin{equation}{section}

\newtheorem{theorem}{Theorem}[section]
\newtheorem{proposition}[theorem]{Proposition}
\newtheorem{remark}[theorem]{Remark}
\newtheorem{lemma}[theorem]{Lemma}
\newtheorem{corollary}[theorem]{Corollary}

\begin{document}
\title[A note on local well-posedness]
{A note on local well-posedness of generalized KdV type equations with dissipative perturbations}

\author[X. Carvajal, M. Panthee]
{Xavier Carvajal, Mahendra Panthee} 

\address{Xavier Carvajal \newline
Instituto de Matem\'atica - UFRJ
Av. Hor\'acio Macedo, Centro de Tecnologia
 Cidade Universit\'aria, Ilha do Fund\~ao,
21941-972 Rio de Janeiro,  RJ, Brasil}
\email{carvajal@im.ufrj.br}

\address{Mahendra Panthee \newline
Departamento de Matem\'atica, UNICAMP\\
Rua Sergio Buarque de Holanda 651\\
13083-859, Campinas, SP, Brazil}
\email{mpanthee@ime.unicamp.br}

\thanks{This work was partially supported by CNPq and  FAPESP, Brazil.}
\subjclass[2010]{35A01, 35Q53}
\keywords{Initial value problem; well-posedness; KdV equation, dispersive-dissipative models}

\begin{abstract}
 In this note we report some local well-posedness results for the Cauchy problems  associated to generalized KdV type equations with dissipative perturbation  for  given data in the low regularity  $L^2$-based Sobolev spaces. The method of proof is based on the {\em contraction mapping principle} employed  in some appropriate time weighted spaces.
\end{abstract}

\maketitle


\allowdisplaybreaks

\section{Introduction}

In this brief note we are interested in reporting well-posedness results for the Cauchy problems associated to generalized Korteweg-de Vries (KdV) type equations with dissipative perturbations, viz., 
\begin{equation}\label{1eq:hs}
 \begin{cases}
  v_t+v_{xxx}+\eta Lv+(v^{k+1})_x=0, \quad x \in \mathbb{R}, \; t\geq 0, \quad k \in \mathbb{N},\\
     v(x,0)=v_0(x),
 \end{cases}
\end{equation}
and
\begin{equation}\label{1eq:hs-1}
 \begin{cases}
  u_t+u_{xxx}+\eta Lu+(u_x)^{k+1}=0, \quad x \in \mathbb{R}, \,t\geq 0,\quad k \in \mathbb{N},\\
  u(x,0)=u_0(x),
 \end{cases}
\end{equation}
where $\eta>0$ is a constant; $u=u(x, t)$, $v=v(x,t)$ are real valued functions and the linear operator $L$ is defined via the Fourier transform by 
$\widehat{Lf}(\xi)=-\Phi(\xi)\widehat{f}(\xi)$. For $p\in \R^+$, the Fourier symbol $\Phi(\xi)\in\R$  is measurable and has the form
\begin{align}\label{phi}
\Phi(\xi)=-|\xi|^p + \Phi_1(\xi),
\end{align}
where $|\Phi_1(\xi)|\leq C(1+|\xi|^q)$ with $0\leq q<p$.
 
The models under consideration are posed for $t\geq 0$. In order to deal with some estimates involving semi-group, we extend these models for all $t\in \R$. For this, we define
\begin{equation}\label{eta}
\eta (t)\equiv \eta \mathop{\rm sgn}(t)= \begin{cases}
 \eta & \text{if } t \geq 0 ,\\
 -\eta & \text{if } t <0,
\end{cases}
\end{equation}
and write \eqref{1eq:hs} and \eqref{1eq:hs-1} respectively in the form
\begin{equation}\label{eq:hs}
   \begin{cases}
    v_t+v_{xxx}+\eta(t) Lv+(v^{k+1})_x=0, \quad x, \; t \in \mathbb{R},\\
     v(0)=v_0,
   \end{cases}
\end{equation}
and
\begin{equation}\label{eq:hs-1}
 \begin{cases}
  u_t+u_{xxx}+\eta(t) Lu+(u_x)^{k+1}=0, \quad x, \; t \in \mathbb{R},\\
  u(x,0)=u_0(x).
 \end{cases}
\end{equation}

This sort of extension to consider the models posed for all $t\in\R$ has already been appeared in earlier works \cite{XC-MP2, XC-MP}. In what follows we will study the well-posedness issues for the Caucy problems \eqref{eq:hs} and \eqref{eq:hs-1}  for given data in the low regularity  Sobolev spaces $H^s(\mathbb{R})$.
We start by introducing two spaces that are suitable to handle the models in question. 

Let  $k\in \mathbb{N} $,  $|t| \leq T\leq 1$, $p> \frac32 k+1$ and $\gamma_k:=\dfrac{3k+2}{2(k+1)}$. To deal with the Cauchy problem  \eqref{eq:hs},  we define, for $s\ge 0$,
\begin{equation}\label{sp.01}
\|f\|_{\mathcal{X}_T^s}:= Ω\sup_{t\in [-T, T]}\Big\{\|f(t)\|_{H^s(\R)}+|t|^{\frac{\gamma_k}{p}}\Big(\|f(t)\|_{L^{2(k+1)}}
+\|\partial_x f(t)\|_{L^{2(k+1)}}+\|D_x^s\partial_xf(t)\|_{L^{2(k+1)}}\Big) \Big\},
\end{equation}
for  $-1<s<0$, 
\begin{equation}\label{sp.01neg}
\|f\|_{\mathcal{X}_T^s}:= Ω\sup_{t\in [-T, T]}\Big\{\|f(t)\|_{H^s(\R)}+|t|^{\frac{\gamma_k}{p}}\Big(\|f(t)\|_{L^{2(k+1)}}+\|D_x^s\partial_xf(t)\|_{L^{2(k+1)}}\Big) \Big\},
\end{equation}
and introduce the Banach space
\begin{equation}\label{bn.1}
\mathcal{X}_T^s:= \big\{ f\in C([-T, T]; H^s(\R)): \|f\|_{\mathcal{X}_T^s} <\infty\big\},
\end{equation}
where  $H^s(\R)$ is the usual $L^2$-based Sobolev space of order $s$.

Similarly, to work on the Cauchy problem \eqref{eq:hs-1}, we define, for $s\in \R$,
\begin{equation}\label{sp.02}
\|f\|_{\mathcal{Y}_T^s} :=\sup_{t\in[-T,T]}\Big\{\| f(t)\|_{H^s}+|t|^{\frac{\gamma_k}{p}}\Big(\|\partial_x f(t)\|_{L^{2(k+1)}}+\|D_x^s\partial_x f(t)\|_{L^{2(k+1)}}\Big)\Big\},
\end{equation}
and introduce, accordingly, the  Banach space
\begin{equation}\label{bn.2}
\mathcal{Y}_T^s:= \big\{  f\in C([-T, T]; H^s(\R)): \|f\|_{\mathcal{Y}_T^s} <\infty\big\}.
\end{equation}

Motivation behind the definition of these sort of spaces is the work of Dix \cite{Dix} where the author uses space with time weight to prove sharp well-posedness result for the Cauchy problem associated to the Burgers' equation  in $H^s(\R)$, $s>-\frac12$ by showing that uniqueness fails for data Sobolev regularity less than $-\frac12$. In our recent work \cite{XC-MP3},  similar spaces are used to get sharp local well-posedness results for KdV type equations. We notice that the space used in \cite{XC-MP3} needs restriction of the Sobolev regularity of the initial data not only from below but also from above. To overcome the restriction of Sobolev regularity from above, we used extra norms in the definitions of $\mathcal{X}_T^s$ and $\mathcal{Y}_T^s$ 	above. Although we could remove the regularity restriction from above with the introduction of extra norms, we could  not lower the regularity requirement of the initial data as much as we desired. This question will be addressed elsewhere.

Now, we state the main results of this work.
\begin{theorem}\label{teorp-1}
 Let $\eta>0$ be fixed, $k>0$ and $\Phi(\xi)$ be as given by \eqref{phi} with $p> \frac32k+1$ as the order of the leading term. Then  
the Cauchy problem \eqref{eq:hs} is locally well-posed for any data $v_0 \in H^s(\mathbb{R})$, whenever $s\geq -1$.
Moreover, the map $v_0\mapsto v$ is smooth from $H^s(\R)$ to $C([-T, T]; H^s(\R))\cap \mathcal{X}_T^s$ and $v\in C([-T, T]\backslash\{0\}; H^{\infty}(\R))$.
\end{theorem}

\begin{theorem}\label{teorp}
Let $\eta>0$ be fixed, $k>0$ and $\Phi(\xi)$ be as given by \eqref{phi} with $p> \frac32k+1$ as the order of the leading term. Then  
the Cauchy problem \eqref{eq:hs-1}  is locally well-posed for any data $u_0 \in H^s(\mathbb{R})$, whenever $s\geq 0$. Moreover, the map $u_0\mapsto u$ is smooth from $H^s(\R)$ to $C([-T, T]; H^s(\R))\cap \mathcal{Y}_T^s$ and $u\in C([-T, T]\backslash\{0\}; H^{\infty}(\R))$.
\end{theorem}

\begin{remark}
The results in Theorems \ref{teorp-1} and \ref{teorp} improve the known well-posedness results for the Cauchy problems \eqref{eq:hs} and \eqref{eq:hs-1} proved in \cite{XC-MP2}. However, in view of our recent work \cite{XC-MP3}, we believe that there is still room to get better results than the ones presented here. Due to the presence of higher order nonlinearity, the time weighted spaces  introduced in \cite{XC-MP3} are not good enough in this case. So, one expects to introduce  some new sort of spaces to cater higher order nonlinearity. Currently, we are working in this direction.
\end{remark}

In earlier works \cite{XC-MP2} (see also \cite{XC-MP}) we proved that the IVPs \eqref{eq:hs} and \eqref{eq:hs-1}  (with $k=1$) for given data in $H^s(\R)$ are locally well-posed whenever $s>-\frac34$ and $s>\frac14$ respectively. The IVPs \eqref{eq:hs} and \eqref{eq:hs-1} with $k= 2, 3, 4$ were also considered in  \cite{XC-MP2} to get local well-posedness results respectively for $s  >\frac14, -\frac16, 0$ and $s  >\frac54, \frac56, 1$. To obtain these results we followed the techniques developed by Bourgain \cite{bou:bou} and Kenig, Ponce and Vega \cite{kpv2:kpv2} (see also \cite{tao}). The main ingredients in the proof  are estimates of the integral equation associated to an extended IVP that is defined for all 
$t\in \mathbb{R}$. The principal tool in \cite{XC-MP2} and \cite{XC-MP} is the  usual Bourgain space associated to  the KdV equation 
instead of that associated to the linear part of the IVPs \eqref{eq:hs} and 
\eqref{eq:hs-1}.

A detailed explanation about the particular examples that belong to the class considered in \eqref{eq:hs} and \eqref{eq:hs-1} and the known well-posedness results about them are presented in our earlier works \cite{XC-MP3, XC-MP2}. 

This paper is organized as follows: In Section \ref{sec-2}, we prove 
some linear and nonlinear estimates.  Section \ref{sec-3} is dedicated to prove the main results. 

\section{Basic estimates}\label{sec-2}
In this section we derive some linear and nonlinear estimates that will be used to prove the main results of this work. Consider the Cauchy problem associated to the linear parts of  \eqref{eq:hs} and \eqref{eq:hs-1},
\begin{equation}\label{eq0.5}
   \begin{cases}
     w_t+w_{xxx}+\eta(t) Lw=0, \quad x, \; t\in \R,\\
     w(0)=w_0.
   \end{cases}
\end{equation}
The solution to \eqref{eq0.5} is given by $w(x,t)=V(t)w_0(x)$
where the  semi-group $V(t)$ is defined as
\begin{equation}\label{gV}
\widehat{V(t)w_0}(\xi)=e^{it\xi^3+\eta |t|\Phi(\xi)}\widehat{w_0}(\xi).
\end{equation}
In what follows, we prove some estimates satisfied by the semi-group defined in \eqref{gV}. Without loss of generality, we suppose $\eta =1$ from here onwards. We start with the following lemmas.
\begin{lemma}[Hausdorff-Young]\label{H-Y} Let $p_1\ge 2$ and $\frac1{p_1}+\frac1{q_1}=1$. Then
\begin{equation}\label{hy-1}
\|f\|_{L^{p_1}}\le C \|\widehat{f}\|_{L^{q_1}}.
\end{equation}
\end{lemma}

\begin{lemma}\label{xav4}
There exists a sufficiently large number $M>0$ such that the following estimates
\begin{equation}\label{xavi9}
\Phi(\xi)=-|\xi|^{p}+\Phi_1(\xi) <-1,
\end{equation}
\begin{equation}\label{eq0.6}
\frac{|\Phi_1(\xi)|}{|\xi|^p}\leq \frac12,
\end{equation}
and 
\begin{equation}\label{eq0.6x34}
|\Phi(\xi)| \ge  \frac{|\xi|^p}2,
\end{equation}
hold true for all   $|\xi| \geq M$.
\end{lemma}

\begin{proof}
Note that,
\begin{equation}\label{eq.m100}
\lim_{|\xi| \to \infty} \frac{\Phi_1(\xi) +1}{|\xi|^p}=0 \quad {\textrm{and}}\quad \lim_{|\xi| \to \infty} \frac{|\Phi_1(\xi)|}{ |\xi|^p}=0.
\end{equation}
 The estimates \eqref{xavi9} and \eqref{eq0.6} follow respectively from the first and the second relation in \eqref{eq.m100}.

For $|\xi|\geq M$, using \eqref{xavi9} and \eqref{eq0.6}, one has that
\begin{equation}
|\Phi(\xi)|= |\xi|^p-\Phi_1(\xi) \ge \dfrac{|\xi|^p}{2},
\end{equation}
which proves \eqref{eq0.6x34}.
\end{proof}

\begin{lemma}\label{lem2.2}
Let $|t|\leq  T$ and $\Phi(\xi)$ be as defined in \eqref{phi}. Then $\Phi(\xi)$ is bounded from above and 
\begin{equation}\label{eq0.11}
\|e^{|t|\Phi(\xi)}\|_{L^{\infty}} \leq e^{TC_M}.
\end{equation}
\end{lemma}

\begin{proof}
Using \eqref{xavi9} in Lemma \ref{xav4}, we see that  there is $M>1$ large enough such that $\Phi(\xi)<-1$. Therefore,
\begin{equation}\label{eq.m101}
e^{|t|\Phi(\xi)} \leq e^{-|t|}\leq 1, \qquad \forall\;\;|\xi|\geq M.
\end{equation} 
 For $|\xi| <M$, it is easy to show that $\Phi(\xi) <C_M$, and consequently 
\begin{equation}\label{eq.m102}
e^{|t|\Phi(\xi)} \leq e^{TC_M}\quad \forall\;\; |\xi| <M.
\end{equation}
 Combining \eqref{eq.m101} and \eqref{eq.m102} we conclude the proof.
\end{proof}

\begin{lemma}\label{lem-smooth} 
Let $v_0\in H^s(\R)$ and $V(t)$ be defined as in \eqref{gV}, then 
$$ V(\cdot)v_0\in C(\R; H^s(\R))\cap C(\R\backslash\{0\}; H^{\infty}(\R)).$$
\end{lemma}
\begin{proof}
It is enough to show that $V(t)v_0\in H^{s'}(\R)$ whenever $s'>s$. We have,
\begin{equation}\label{sm0.1}
\begin{split}
\|V(t)v_0\|_{H^{s'}} &=\|\langle\xi\rangle^{s'}e^{it\xi^3+|t|\Phi(\xi)}\widehat{v_0}(\xi)\|_{L^2}\\
&=\|\langle\xi\rangle^{s}\widehat{v_0}(\xi)\langle\xi\rangle^{s'-s}e^{|t|\Phi(\xi)}\|_{L^2}\\
&\leq \|\langle\xi\rangle^{s'-s}e^{|t|\Phi(\xi)}\|_{L^{\infty}}\|v_0\|_{H^s},
\end{split}
\end{equation}
where we have used the notation $\langle x\rangle :=1+|x|$. 

Let $M\gg1$ be as in Lemma \ref{xav4}, then one gets
\begin{equation}\label{sm0.2}
\begin{split}
\|\langle\xi\rangle^{s'-s}e^{|t|\Phi(\xi)}\|_{L^{\infty}}
&\leq \|\langle\xi\rangle^{s'-s}e^{|t|\Phi(\xi)}\|_{L^{\infty}(|\xi|\leq M)}
+\|\langle\xi\rangle^{s'-s}e^{|t|\Phi(\xi)}\|_{L^{\infty}(|\xi|> M)}\\
&\leq  C_M + \|e^{-\frac{|\xi|^p}2|t|}\langle\xi\rangle^{s'-s}\|_{L^{\infty}} <\infty, \qquad t\in \R\backslash\{0\}.
\end{split}
\end{equation}
One can use the dominated convergence theorem to prove continuity.
\end{proof}

\begin{lemma}\label{linear1x} Let $s>-1-\frac{k}{2(k+1)}$, $|t|\le T \le 1$ and $p> 0$, then we have
\begin{equation}\label{cp.001}
|t|^{\frac{\gamma_k}{p}}\|\partial_x D_x^s V(t)w_0\|_{L^{2(k+1)}} \le C \|w_0\|_{H^s}
\end{equation}
and
\begin{equation}\label{cp.002}
|t|^{\frac{\gamma_k}{p}}\|V(t)w_0\|_{L^{2(k+1)}} \le C \|w_0\|_{H^{-1}}.
\end{equation}
\end{lemma}

\begin{proof}
By Hausdorff-Young and H\"older inequalities, we obtain
\begin{equation}\label{lemnorm1}
\|\partial_x  D_x^s V(t)w_0\|_{L^{2(k+1)}} \le C \|e^{|t|\Phi(\xi)} \xi |\xi|^s \widehat{w_0}\|_{L^{\frac{2(k+1)}{2k+1}}} \le C\Big\|e^{|t|\Phi(\xi)}  \dfrac{\xi |\xi|^s}{\langle \xi \rangle^s}\Big\|_{L^{\frac{2(k+1)}{k}}} \| \langle \xi \rangle^s\widehat{w_0}\|_{L^2}.
\end{equation}
Let $M$ as in Lemma \ref{xav4}, then we have
\begin{equation}\label{m0.01sma0}
\begin{split}
\Big\| \dfrac{\xi |\xi|^s}{\langle \xi \rangle^s} e^{|t|\Phi(\xi)}\Big\|_{L^{\frac{2(k+1)}{k}}} &\leq \Big\| \dfrac{\xi |\xi|^s}{\langle \xi \rangle^s} e^{|t|\Phi(\xi)}\chi_{\{|\xi|\leq M\}}\Big\|_{L^{\frac{2(k+1)}{k}}} +
\Big\|\dfrac{\xi |\xi|^s}{\langle \xi \rangle^s} e^{|t|\Phi(\xi)}\chi_{\{|\xi|> M\}}\Big\|_{L^{\frac{2(k+1)}{k}}}\\
& \le  C_M+ \left( \int_{\R} |\xi|^{\frac{2(k+1)}{k}} e^{-|t|(\frac{k+1}k)|\xi|^p} d \xi \right)^{\frac{k}{2(k+1)}}.
\end{split}
\end{equation}
Using a change of variable $|t|^{-1/p} \xi:=x$, we get
\begin{equation}\label{m0.01sma01}
\begin{split}
\Big\|\dfrac{\xi |\xi|^s}{\langle \xi \rangle^s} e^{|t|\Phi(\xi)}\Big\|_{L^{\frac{2(k+1)}{k}}} & \le  C_M+ \dfrac{1}{|t|^{\frac{\gamma_k}{p}}}
\left( \int_{\R} |x|^{\frac{2(k+1)}{k}} e^{-(\frac{k+1}k)|x|^p} d x \right)^{\frac{k}{2(k+1)}}.
\end{split}
\end{equation}

Since the integral in \eqref{m0.01sma01} is finite, inserting \eqref{m0.01sma01} in \eqref{lemnorm1} and multiplying by $|t|^{\frac{\gamma_k}{p}}$ we get the required estimate  \eqref{cp.001}. 

With a similar argument as above (considering $s=0$ in \eqref{cp.001} ), one can obtain
\begin{equation}\label{xxx1}
|t|^{\frac{\gamma_k}{p}}\|D_x  V(t)w_0\|_{L^{2(k+1)}} \le C \|w_0\|_{L^2}.
\end{equation}
Analogously, we can  prove
\begin{equation}\label{xxx2}
|t|^{\frac{\gamma_k}{p}}\| V(t)w_0\|_{L^{2(k+1)}} \le C \|w_0\|_{L^2}.
\end{equation}
Now, using \eqref{xxx1} and \eqref{xxx2}, we obtain 
\begin{equation}\label{rev-01}
|t|^{\frac{\gamma_k}{p}}\| (1+D_x)V(t)w_0 \|\le C \|w_0\|_{L^2}.
\end{equation}
This final estimate \eqref{rev-01}  is equivalent to \eqref{cp.002} and this completes the proof.
\end{proof}

\begin{corollary}\label{cor-rev-1}
 Let $s>-1-\frac{k}{2(k+1)}$,\,  $|\tau|\leq T$, \, $|t|\le T \le 1$ and $p> 0$, then we have
\begin{equation}\label{cp.001}
|t-\tau|^{\frac{\gamma_k}{p}}\|\partial_x D_x^s V(t-\tau)f(t, \cdot)\|_{L^{2(k+1)}} \le C \|f(t, \cdot)\|_{H^s}
\end{equation}
and
\begin{equation}\label{cp.002}
|t-\tau|^{\frac{\gamma_k}{p}}\|V(t-\tau)f(t, \cdot)\|_{L^{2(k+1)}} \le C \|f(t, \cdot)\|_{H^{-1}}.
\end{equation}
\end{corollary}

\begin{lemma}\label{mm.1}
Let $V(t)$ be as defined in \eqref{gV}, $|t|\le T \le 1$ and $p> \frac32 k+1$  then for all $s\geq -1$, we have
\begin{equation}\label{mm.2}
\|V(t)w_0\|_{\mathcal{X}_T^s} \leq C\|w_0\|_{H^s}
\end{equation}
and for all $s\geq 0$, we have
\begin{equation}\label{linear1}
\|V(t)w_0\|_{\mathcal{Y}^s_T} \le C \|w_0\|_{H^s}.
\end{equation}
\end{lemma}

\begin{proof}
We provide a detailed proof for \eqref{mm.2}, the proof of \eqref{linear1} follows similarly. We start with the first component of $\mathcal{X}_T^s$-norm. First, note that
\begin{equation}\label{eq0.10}
\|V(t)w_0\|_{H^s} = \|\langle\xi\rangle^se^{|t|\Phi(\xi)}\widehat{w_0}(\xi)\|_{L^2} \leq \|e^{|t|\Phi(\xi)}\|_{L^{\infty}}\|w_0\|_{H^s}.
\end{equation}
From \eqref{eq0.11} and \eqref{eq0.10}, we can conclude
\begin{equation}\label{cp.1}
\|V(t)w_0\|_{H^s} \leq e^{TC_M}\|w_0\|_{H^s}.
\end{equation}

For the next components of the $\mathcal{X}_T^s$-norm, we use \eqref{cp.001} and \eqref{cp.002} from Lemma \ref{linear1x} to get, for $s\geq -1$
\begin{equation}\label{cp.3}
|t|^{\frac{\gamma_k}p}\|D_x^s \partial_x V(t)w_0\|_{L^{2(k+1)}}+|t|^{\frac{\gamma_k}p}\|V(t)w_0\|_{L^{2(k+1)}} \leq C\|w_0\|_{H^{s}}.
\end{equation}

Combining \eqref{cp.1} and \eqref{cp.3} we get the required estimate \eqref{mm.2}. 
\end{proof}

\begin{lemma}\label{mm.03}
Let $V(t)$ be as defined in \eqref{gV}, $ |t|\le T \le 1$ and $p> \frac32 k+1$,  then the following estimates hold true.
\begin{equation}\label{mm.04}
\Big\|\int_0^tV(t-\tau) \partial_x(v^{k+1})(\tau)d\tau\Big\|_{{\mathcal{X}_T^s}} \leq T^{\omega_k}\|v\|_{\mathcal{X}_T^s}^{k+1}, \quad \forall\: s\geq -1,
\end{equation}
and
\begin{equation}\label{mm.05}
\Big\|\int_0^tV(t-\tau) (\partial_xu)^{k+1}(\tau)d\tau\Big\|_{{\mathcal{Y}_T^s}} \leq T^{\omega_k}\|u\|_{\mathcal{Y}_T^s}^{k+1} \quad \forall\: s\geq 0,
\end{equation}
where $\omega_k=\frac{2p-3k-2}{2p}>0$.
\end{lemma}
\begin{proof}
We start by  proving \eqref{mm.04}. First consider the case when $s\ge0$.
Using the definition of  $\mathcal{X}_T^s$--norm and  Lemma \ref{linear1x}  for $s\ge0$ (see \eqref{sp.01}), we get, for the first component
\begin{equation}\label{rev-1}
  \sup_{t\in[-T, T]}\Big\|\int_0^tV(t-\tau) \partial_x(v^{k+1})(\tau)d\tau\Big\|_{{{H}_T^s}}\leq   \sup_{t\in[-T, T]}\int_0^{|t|}\|\partial_x(v^{k+1})(\tau)\|_{H^s}d\tau.
  \end{equation}
 To estimate  the other components of  $\mathcal{X}_T^s$--norm, we use Minkowski's integral inequality and  Corollary \ref{cor-rev-1}. For the sake of brevity and clarity, we give details for the  fourth component only, because others follow similarly.
  \begin{equation}\label{rev-02}
\begin{split}
  |t|^{\frac{\gamma_k}p}\Big\|D_x^s\partial_x\int_0^tV(t-\tau) \partial_x(v^{k+1})(\tau) &d\tau\Big\|_{L^{2(k+1)}}
  = |t|^{\frac{\gamma_k}p}\Big\|\int_0^tD_x^s\partial_xV(t-\tau) \partial_x(v^{k+1})(\tau)d\tau\Big\|_{L^{2(k+1)}}\\
	&\leq C\,|t|^{\frac{\gamma_k}p}\int_0^{|t|}\Big\|D_x^s\partial_xV(t-\tau) \partial_x(v^{k+1})(\tau)\Big\|_{L^{2(k+1)}}d\tau\\
	&\leq C \,|t|^{\frac{\gamma_k}p} \int_0^{|t|}  \dfrac{1}{\,|t-\tau|^{\frac{\gamma_k}p}}\|\partial_x(v^{k+1})(\tau)\|_{H^s}d\tau.
\end{split}
\end{equation}

 Combining \eqref{rev-1},  \eqref{rev-02} and similar estimates for the other components,  we obtain 
\begin{equation}\label{mmxx.06}
\begin{split}
  \Big \|\int_0^t &V(t-\tau) \partial_x(v^{k+1})(\tau)d\tau\Big\|_{{\mathcal{X}_T^s}}
  \leq \sup_{t\in[-T, T]}\int_0^{|t|} \left(\| \partial_x(v^{k+1})\|_{L^2} +\|D_x^s \partial_x(v^{k+1})\|_{L^2} \right) d\tau\\
  & +C\sup_{t\in[-T, T]} |t|^{\frac{\gamma_k}p} \int_0^{|t|}  \dfrac{1}{\,|t-\tau|^{\frac{\gamma_k}p}} \left(\| \partial_x(v^{k+1})\|_{L^2} +\|D_x^s \partial_x(v^{k+1})\|_{L^2} \right)d\tau\\
	&\leq C\!\!\sup_{t\in[-T, T]}|t|^{\frac{\gamma_k}p}\!\left[\int_0^{|t|} \!\! \dfrac{1}{\,|t-\tau|^{\frac{\gamma_k}p}} \| \partial_x(v^{k+1})\|_{L^2}d\tau 
	+\int_0^{|t|}  \dfrac{1}{\,|t-\tau|^{\frac{\gamma_k}p}}\|D_x^{\tilde{s}} (v^{k+1})\|_{L^2} d\tau\right]\\
	&=:I_1+I_2,
\end{split}
\end{equation}
	where we used  $\tilde{s} =1+s$ and the estimate
$$
1=\dfrac{|t-\tau|}{|t-\tau|}	\leq \dfrac{|t|+|\tau|}{|t-\tau|} \leq \dfrac{2|t|}{|t-\tau|}, \quad 0\leq \tau <|t|.
$$

	In what follows, we will obtain an estimate for $I_2$. 
	Now, using fractional chain rule (see Tao \cite{Tao1} (A.15) page 338), for $\tilde{s} \geq 0$, we have
	\begin{equation}\label{mx.001}
	\begin{split}
I_2& \leq C\!\!\sup_{t\in[-T, T]}|t|^{\frac{\gamma_k}p}\int_0^{|t|}  \dfrac{1}{\,|t-\tau|^{\frac{\gamma_k}p}} \|v^k\|_{L^{\frac{2(k+1)}{k}}}\|D_x^{\tilde{s}} v\|_{L^{2(k+1)}} d\tau\\
& \leq C\,\| v\|_{\mathcal{X}_T^s}^{k+1}\sup_{t\in[-T, T]}|t|^{\frac{\gamma_k}p}\int_0^{|t|} \dfrac{1}{\,|t-\tau|^{\frac{\gamma_k}p}} \frac{1}{\tau^{\frac{k\gamma_k}p}}  \frac{1}{\tau^{\frac{\gamma_k}p}}d\tau\\
& \leq C\,\| v\|_{\mathcal{X}_T^s}^{k+1}\sup_{t\in[-T, T]}|t|^{\frac{\gamma_k}p}\int_0^{|t|}  \dfrac{1}{\,|t-\tau|^{\frac{\gamma_k}p}} \frac{1}{\tau^{\frac{3k+2}{2p}}}d\tau \\
& \leq C_{p,k} \,\| v\|_{\mathcal{X}_T^s}^{k+1} T^{\omega_k}.
\end{split}
\end{equation}

The estimate for $I_1$ will follow from the one of $I_2$ by considering $s=0$. In fact,
\begin{equation}\label{mx.002}
	\begin{split}
I_1& \leq C\!\!\sup_{t\in[-T, T]}|t|^{\frac{\gamma_k}p}\int_0^t  \dfrac{1}{\,|t-\tau|^{\frac{\gamma_k}p}} \|v^k\|_{L^{\frac{2(k+1)}{k}}}\|\partial_x v\|_{L^{2(k+1)}} d\tau\\
& \leq C\,\| v\|_{\mathcal{X}_T^s}^{k+1} \sup_{t\in[-T, T]}|t|^{\frac{\gamma_k}p} \int_0^t  \dfrac{1}{\,|t-\tau|^{\frac{\gamma_k}p}} \frac{1}{\tau^{\frac{3k+2}{2p}}}d\tau \\
&\leq C_{p,k} \,\| v\|_{\mathcal{X}_T^s}^{k+1} T^{\omega_k}.
\end{split}
\end{equation}

Inserting estimates \eqref{mx.001} and \eqref{mx.002} in \eqref{mmxx.06} we  obtain the required estimate \eqref{mm.04} in the case $s\geq 0$.

 Next we consider the case $-1\le s<0$.  Using a similar argument as above, one gets
\begin{equation}\label{mmxx.06neg}
\begin{split}
  \Big\|\int_0^tV(t-\tau) \partial_x(v^{k+1})(\tau)d\tau\Big\|_{{\mathcal{X}_T^s}}
  &\leq C\!\! \sup_{t\in[-T, T]}|t|^{\frac{\gamma_k}p} \int_0^{|t|} \dfrac{1}{\,|t-\tau|^{\frac{\gamma_k}p}}\|\partial_x(v^{k+1})(\tau)\|_{H^s}d\tau\\
	&\leq C\!\! \sup_{t\in[-T, T]}|t|^{\frac{\gamma_k}p} \int_0^{|t|}  \dfrac{1}{\,|t-\tau|^{\frac{\gamma_k}p}}\|D_x^s \partial_x(v^{k+1})\|_{L^2} d\tau\\
	&\leq C\!\! \sup_{t\in[-T, T]}|t|^{\frac{\gamma_k}p} \int_0^{|t|}  \dfrac{1}{\,|t-\tau|^{\frac{\gamma_k}p}}\|D_x^{\tilde{s}} (v^{k+1})\|_{L^2} d\tau\\
	& \leq C_{p,k} \,\| v\|_{\mathcal{X}_T^s}^{k+1} T^{\omega_k},
\end{split}
\end{equation}
where $\tilde{s} =1+s \ge  0$, for $s\geq -1$.

 Now, we move to prove the estimate \eqref{mm.05}. By using \eqref{linear1} from Lemma \ref{mm.1} and fractional chain rule as in \eqref{mx.001}, for $s\geq 0$, we get
\begin{equation}\label{mm.06}
\begin{split}
 \Big\|\int_0^t &V(t-\tau) (\partial_xu)^{k+1}(\tau)d\tau\Big\|_{{\mathcal{Y}_T^s}}\\
&\leq C\!\!\sup_{t\in[-T, T]}|t|^{\frac{\gamma_k}p}\int_0^{|t| }  \dfrac{1}{\,|t-\tau|^{\frac{\gamma_k}p}} \left(\| (u_x)^{k+1}\|_{L^2} +\|D_x^s (u_x)^{k+1}\|_{L^2} \right) d\tau\\
& \leq C\!\!\sup_{t\in[-T, T]}|t|^{\frac{\gamma_k}p}\int_0^{|t|}  \dfrac{1}{\,|t-\tau|^{\frac{\gamma_k}p}}  \left(\| u_x\|_{L^{2(k+1)}}^{k+1} +\| (u_x)^{k}\|_{L^{\frac{2(k+1)}{k}}}\|D_x^s (u_x)\|_{L^{2(k+1)}} \right) d\tau\\
& \leq C\| u\|_{\mathcal{Y}_T^s}^{k+1}\sup_{t\in[-T, T]}|t|^{\frac{\gamma_k}p}\int_0^{|t|}  \dfrac{1}{\,|t-\tau|^{\frac{\gamma_k}p}} 
\left[ \frac{1}{\tau^{\frac{{(k+1)}\gamma_k}p}}+\frac{1}{\tau^{\frac{k\gamma_k}p}}\,  \frac{1}{\tau^{\frac{\gamma_k}p}}\right]d\tau\\
& \leq C\| u\|_{\mathcal{Y}_T^s}^{k+1}\sup_{t\in[-T, T]}|t|^{\frac{\gamma_k}p}\int_0^{|t|}    \dfrac{1}{\,|t-\tau|^{\frac{\gamma_k}p}} \frac{1}{\tau^{\frac{3k+2}{2p}}}d\tau \\
& \leq C_{p,k} \,\| u\|_{\mathcal{Y}_T^s}^{k+1} T^{\omega_k},
\end{split}
\end{equation}
as required.  In the last inequality of \eqref{mm.06} the condition $p> \frac32 k+1$ has been used.
\end{proof}

Now, we move to prove some more estimates that are useful in our analysis.

\begin{lemma}\label{xav6.6}
Let $\theta \ge 0$, $p>0$ and $\tau \in [-1, 1]\backslash\{0\}$. Then we have
\begin{equation}\label{eq0.99}
\| \langle \xi \rangle^\theta e^{|\tau| \Phi(\xi)}\|_{L^\infty_\xi}\lesssim  \dfrac{1}{|\tau|^{\frac{\theta}{p}}}.
\end{equation}
\end{lemma}
\begin{proof}

Considering $M$ large as in Lemma \ref{xav4}, one can obtain
\begin{equation}\label{eq2.03}
\|\langle \xi\rangle^\theta e^{|\tau|\Phi(\xi)}\|_{L^\infty}  \le \|\langle \xi\rangle^{\theta} e^{|\tau|\Phi(\xi)}\|_{L^\infty(|\xi| \leq M)} + \|\langle \xi\rangle^{\theta} e^{|\tau|\Phi(\xi)}\|_{L^\infty(|\xi| > M)} =:A_1 +A_2.
\end{equation}

For $\tau\in [-1, 1]\backslash\{0\}$ and  $\frac{\theta}{p}\ge 0$, we have 
\begin{equation}\label{eq2.04}
A_1\leq C_Me^{|\tau| C_M} \lesssim \frac1{|\tau|^{\frac{\theta}{p}}}.
\end{equation}

In what follows, we  obtain and estimate for the high frequency part $A_2$. From  \eqref{eq0.6x34} we  have
\begin{equation}\label{eq2.05}
A_2\lesssim \dfrac{\|\, |\,|\tau|^{\frac{1}{p}}\, \xi|^{\theta} e^{-| |\tau|^{\frac{1}{p}}\,\xi|^p}\|_{L^\infty(|\xi| > M)}}{|\tau|^{\frac{\theta}{p}}}  \lesssim \frac1{|\tau|^{\frac{\theta}{p}}},
\end{equation}
where in the last inequality $x^{\frac{\theta}p}e^{-x} \le C_{\theta, p}$, if $x\ge0$ has been used. The proof of the lemma follows inserting \eqref{eq2.04} and \eqref{eq2.05} in \eqref{eq2.03}.
\end{proof}

\begin{proposition}\label{xav7smoth}
Let $s>-1$, $p>  \frac{3k}2+1$,  $k\in \mathbb{N}$. There exists  $ \mu:=\mu(s,p,k) >0$ such that  if
\begin{equation}\label{xav3smoth}
\|f\|_{\mathcal{Z}_T^s}:=\sup_{t\in[-T,T]\backslash\{0\}}\Big\{\|f(t)\|_{H^s}+t^{\frac{\gamma_k}{p}}\|f(t)\|_{L^{2(k+1)}}\Big\}< \infty, 
\end{equation}
then  the application
\begin{equation}\label{eq0.9x3smoot}
t\mapsto \mathcal{L}(f)(t):=\int_0^t V(t-t')\partial_x(f^{k+1}) (t') dt',\qquad 0< |t|\le T\le 1,
\end{equation}
is continuous from $[-T, T]\backslash\{0\}$ to  $H^{s+\mu}$. 
\end{proposition}

\begin{proof}
We prove this proposition considering  $0<t\leq T \leq 1$. The case $-1\leq -T\leq t <0$ follows similarly. First we show that $\mathcal{L}(f)(t) \in H^{s+\mu}(\R)$ $\forall$ $f\in Z_T^s$.
We do this considering two different cases.

\noindent
{\bf Case I, $-1\le s<p-1$:} Let $0< t\le T\le 1$ and $f\in Z_T^s$. In this case, we have
\begin{equation}\label{eq:2.42}
\begin{split}
\|\mathcal{L}(f)(t)\|_{H^{s+\mu}} &=\Big\|\langle \xi \rangle^{s+\mu}\int_0^{t}\left( e^{i(t-t')\xi^3+|t-t'|\Phi(\xi)} \right) \,i \xi  \widehat{f^{k+1}}(\xi,t')dt'\Big\|_{L^2}\\
&  \le \int_0^{t} \|\langle \xi \rangle^{s+\mu}\,\xi\, e^{|t-t'|\Phi(\xi)}  \|_{L^\infty}  \| \widehat{f^{k+1}}(\cdot,t')\|_{L^2} dt'\\
&  \le \int_0^{t} \|\langle \xi \rangle^{s+\mu+1} e^{|t-t'|\Phi(\xi)}  \|_{L^\infty}  \| f(\cdot,t')\|_{L^{2(k+1)}}^{k+1} dt'
\\
&\lesssim  \|f \|_{\mathcal{Z}_T^s}^{k+1}  \int_0^{t}  \dfrac{1}{|t-t'|^{\frac{1+s+\mu}p}}\,\dfrac{1}{|t'|^{\frac{3k+2}{2p}}}\, dt' <\infty,
\end{split}
\end{equation}
where $0<\mu< p-1-s$, the definition of  $\mathcal{Z}_T^s$-norm, Minkowski's inequality and  inequality \eqref{eq0.99} from Lemma \ref{xav6.6} are used. Note that, the condition $s\geq -1$ guarantees $s+\mu+1\geq 0$ to use \eqref{eq0.99}.

\noindent
{\bf Case II, $s\ge p-1$:}  Similarly as above,  using the fact that, in this case $H^s$ is an algebra, we obtain
\begin{equation}\label{eq:243}
\begin{split}
\|\mathcal{L}(f)(t)\|_{H^{s+\mu}} &=\Big\|\langle \xi \rangle^{s+\mu}\int_0^{t}\left( e^{i(t-t')\xi^3+|t-t'|\Phi(\xi)} \right) \,i \xi \widehat{f^{k+1}}(\xi,t')dt'\Big\|_{L^2}\\
&\le \int_0^{t} \|\langle \xi \rangle^{\mu} \, \xi \,e^{|t-t'|\Phi(\xi)}  \left( J^{s} (f^{k+1})(\cdot,t') \right)^{\wedge}(\xi) \|_{L^2}dt'\\
&  \le \int_0^{t} \|\langle \xi \rangle^{\mu}  \xi \left( e^{|t-t'|\Phi(\xi)} \right) \|_{L^\infty}  \|  J^{s}( f^{k+1})(\cdot,t')\|_{L^{2}} dt'
\\
&  \le \int_0^{t} \|\langle \xi \rangle^{\mu+1}\left( e^{|t-t'|\Phi(\xi)} \right) \|_{L^\infty}  \|f(\cdot,t') \|_{H^s}^{k+1} dt'\\
&\lesssim  \|f \|_{\mathcal{Z}_T^s}^{k+1}  \int_0^{t}  \dfrac{1}{|t-t'|^{\frac{1+\mu}p}}\, dt' <\infty.
\end{split}
\end{equation}
The last estimate in \eqref{eq:243} follows by choosing $0<\mu<p-1$.

To prove the continuity part,  let $t_0\in (0,T]$ fixed,  and let $f\in  Z_T^s$. We will prove that
\begin{equation}
\lim_{t \to t_0}\|\mathcal{L}(f)(t)-\mathcal{L}(f)(t_0)\|_{H^{s+\mu}}=0.
\end{equation}
We divide the proof in two different cases.

\noindent
{\bf Case a, $0<t\leq t_0$:} In this case, using \eqref{eq0.9x3smoot} and the additive property of the integral, we obtain
\begin{equation}\label{gmp-1}
\begin{split}
\|\mathcal{L}(f)(t)&-\mathcal{L}(f)(t_0)\|_{H^{s+\mu}}
=\Big\|\int_0^{t_0} V(t_0-t')\partial_x(f^{k+1})(t')d\tau-\int_0^t V(t-t')\partial_x(f^{k+1})(t')dt' \Big\|_{H^{s+\mu}}\\
&\le  
\Big\|\int_0^{t}\left( V(t_0-t')-V(t-t') \right)\partial_x(f^{k+1})(t')dt'\Big\|_{H^{s+\mu}}+\Big\|\int_{t}^{t_0} V(t_0-t') \partial_x(f^{k+1})(t')dt'\Big\|_{H^{s+\mu}}\\
&=: I_1(t,t_0)+I_2(t,t_0).
\end{split}
\end{equation}

We write 
\begin{equation}
\begin{split}
I_1(t,t_0)= & \Big\|\langle \xi \rangle^{s+\mu}\int_0^{t}\left( e^{i(t_0-t')\xi^3+(t_0-t')\Phi(\xi)}-e^{i(t-t')\xi^3+(t-t')\Phi(\xi)} \right) \,i \xi \widehat{f^{k+1}}(\xi,t')dt'\Big\|_{L^2},
\end{split}
\end{equation}
and note that  $$( e^{i(t_0-t')\xi^3+(t_0-t')\Phi(\xi)}-e^{i(t-t')\xi^3+(t-t') \Phi(\xi)}) =\int_t^{t_0} (i\xi^3+ \Phi(\xi) )e^{i(\tau-t')\xi^3+(\tau-t')\Phi(\xi)} d\tau.$$

We analyse $I_1$, considering two different cases.

\noindent
{\bf Case a.1, $-1\le s< \min\{ 2(p-2), p-1\}$:}  In this case, we have
\begin{equation}\label{gn-1}
\begin{split}
I_1(t,t_0)= &\Big \|  \int _0^t \int_t^{t_0}\langle \xi \rangle^{s+\mu}(i\xi^3+ \Phi(\xi) )e^{i(\tau-t')\xi^3+(\tau-t')\Phi(\xi)} \xi \widehat{f^{k+1}}(\xi,t')d\tau dt'\Big\|_{L^2}\\
& \le \Big\|  \int _0^t \int_t^{t_0} \langle \xi \rangle^{s+\mu+1}|i\xi^3+ \Phi(\xi) |e^{(\tau-t')\Phi(\xi)} |\widehat{f^{k+1}}(\xi,t')| d\tau dt'\Big\|_{L^2} \\
& \le \int _0^t \int_t^{t_0}\|  \langle \xi \rangle^{s+\mu+1}|i\xi^3+ \Phi(\xi) |e^{(\tau-t')\Phi(\xi)} \widehat{f^{k+1}}(\xi,t')  \|_{L^2} d\tau dt'\\ 
&\lesssim \int _0^t \int_t^{t_0} \|  \langle \xi \rangle^{s+\mu+r}e^{(\tau-t')\Phi(\xi)} \widehat{f^{k+1}}(\xi,t')  \|_{L^2}d\tau dt'\\
& \lesssim  \|f \|_{\mathcal{Z}_T^s}^{k+1} \int _0^t  \int_t^{t_0}  \dfrac{1}{|\tau-t'|^{\frac{\mu+s+r}p}}\,\dfrac{1}{|t'|^{\frac{3k+2}{2p}}}\, d\tau dt'\\
& =:   \|f \|_{\mathcal{Z}_T^s}^{k+1}J,
\end{split}
\end{equation}
where $r=\max\{3,p\}+1$. In the domain of integration one has $0\le t' \le t \le \tau\le t_0$, and consequently $|\tau -t'| =(\tau-t')\ge t-t'=|t-t'|$. Therefore, we can estimate $J$ as follows
\begin{equation}\label{gn-2}
\begin{split}
J&=\int _0^t \,\dfrac{1}{|t'|^{\frac{3k+2}{2p}}}\, \int_t^{t_0}  \dfrac{1}{|\tau-t'|^{\frac{\mu+s+r}p}}\, d\tau dt'\\
& \sim \int _0^t \,\dfrac{1}{|t'|^{\frac{3k+2}{2p}}}\,\left( \dfrac{1}{(t_0-t')^{\alpha_{\mu}}}- \dfrac{1}{(t-t')^{\alpha_{\mu}}} \right) dt',
\end{split}
\end{equation}
where
\begin{equation}\label{gn-3}
\alpha_{\mu}=\alpha_{\mu}(s,p)=
\begin{cases}
\frac{\mu+s+1}{p}, \quad\quad&\mathrm{ if} \quad p\geq 3 \quad \mathrm{and}\quad 0<\mu<p-1-s,\\
\frac{\mu+s+4-p}{p}, \quad &{\mathrm{if}} \quad  p\leq 3\quad \mathrm{and}\quad  0<\mu<2(p-2)-s.
\end{cases}
\end{equation}
Note that, for the choices of $p$, $s$ and $\mu$, we have $ \alpha_{\mu} <1$.

Now, making change of variables $t'=t_0x$ and $t'=tx$ respectively in the first and second integrals in \eqref{gn-2} and taking in account that $\alpha_{\mu}<1$,  we obtain
\begin{equation}\label{gn-33}
\begin{split}
J & \sim  t_0^{1-\alpha_{\mu}-\frac{3k+2}{2p}} \int _0^{t/t_0} \,\dfrac{1}{|x|^{\frac{3k+2}{2p}}}\, \dfrac{1}{(1-x)^{\alpha_{\mu}}} dx- t^{1-\alpha_{\mu}-\frac{3k+2}{2p}} \int _0^{1} \,\dfrac{1}{|x|^{\frac{3k+2}{2p}}}\, \dfrac{1}{(1-x)^{\alpha_{\mu}}} dx \to 0,
\end{split}
\end{equation}
 whenever $t \to t_0$.

\noindent
{\bf Case a.2, $s\ge   \min\{ 2(p-2), p-1\}$:}  In this case, with a similar argument as above,  and the fact that $H^s$ is an algebra, one can get
\begin{equation}\label{gn-4}
\begin{split}
I_1(t,t_0)
& \lesssim  \|f \|_{\mathcal{Z}_T^s}^{k+1} \int _0^t  \int_t^{t_0}  \dfrac{1}{|\tau-t'|^{\frac{\mu+r}p}}\, d\tau dt'\\
& :=   \|f \|_{\mathcal{Z}_T^s}^{k+1}\tilde{J}.
\end{split}
\end{equation}

Similarly to $J$ in \eqref{gn-2}, we get
\begin{equation}\label{gn-5}
\begin{split}
\tilde{J} & \sim \int _0^t \,\left( \dfrac{1}{(t_0-t')^{\alpha_{\mu}}}- \dfrac{1}{(t-t')^{\alpha_{\mu}}} \right) dt',
\end{split}
\end{equation}
where
\begin{equation}\label{g-6}
\alpha_{\mu}=\alpha_{\mu}(s,p)=
\begin{cases}
\frac{\mu+1}{p}, \quad \quad&\mathrm{ if}  \quad p\geq 3 \quad \mathrm{and}\quad 0<\mu<p-1,\\
\frac{\mu+4-p}{p}, \quad&\mathrm{ if}  \quad   p\leq 3 \quad \mathrm{and}\quad 0<\mu<2(p-2).
\end{cases}
\end{equation}
In this case too, for the choices of $p$, $s$ and $\mu$, we have that $ \alpha_{\mu} <1$.

As in \eqref{gn-33}, making change of variables and the fact that $\alpha_{\mu}<1$, we obtain
\begin{equation}\label{ge-7}
\begin{split}
\tilde{J}
& \sim  t_0^{1-\alpha_{\mu}} \int _0^{t/t_0} \, \dfrac{1}{(1-x)^{\alpha_{\mu}}} dx- t^{1-\alpha_{\mu}} \int _0^{1} \, \dfrac{1}{(1-x)^{\alpha_{\mu}}} dx \to 0,
\end{split}
\end{equation}
 whenever $t \to t_0$.
 
 Therefore, in the light of \eqref{gn-33} and \eqref{ge-7}, we get
 \begin{equation}\label{gn-8}
I_1(t,t_0) \to 0, \qquad \textrm{if}\quad t\to t_0.
\end{equation}
Analogously, since
$$
\int_{0}^{t_0} \|V(t-t') \partial_x(f^{k+1})(t')dt'\|_{H^{s+\mu}} <\infty
$$
we get
\begin{equation}\label{gn-9}
I_2(t,t_0) \to 0, \qquad \textrm{if}\quad t\to t_0.
\end{equation}

Hence, using \eqref{gn-8} and \eqref{gn-9} in \eqref{gmp-1}, we conclude the proof of the proposition in this case.

\noindent
{\bf Case b, $0<t_0< t$:} In this case too, using \eqref{eq0.9x3smoot} and the additive property of the integral, we obtain
\begin{equation}
\begin{split}
\|\mathcal{L}(f)(t)&-\mathcal{L}(f)(t_0)\|_{H^{s+\mu}}
=\Big\|\int_0^{t_0} V(t_0-t')\partial_x(f^{k+1})(t')d\tau-\int_0^t V(t-t')\partial_x(f^{k+1})(t')dt' \Big\|_{H^{s+\mu}}\\
&\le  
\Big\|\int_0^{t_0}\left( V(t_0-t')-V(t-t') \right)\partial_x(f^{k+1})(t')dt'\Big\|_{H^{s+\mu}}+\Big\|\int_{t_0}^{t} V(t-t') \partial_x(f^{k+1})(t')dt'\Big\|_{H^{s+\mu}}\\
&=: I_1(t,t_0)+I_2(t,t_0).
\end{split}
\end{equation}
The rest follows analogously  as in {\bf Case a}, and this completes the proof of the proposition.
\end{proof}

\begin{proposition}\label{xav7sm}
Let $s\ge 0$, $p> \frac{3k}2+1$. There exists  $ \mu:=\mu(s,p,k) >0$ such that  if
\begin{equation}\label{xav3sm}
\|f\|_{\tilde{\mathcal{Z}}_T^s}:=\sup_{t\in[-T,T]\backslash\{0\}}\Big\{\|f(t)\|_{H^s}+t^{\frac{\gamma_k}{p}}\|\partial_xf(t)\|_{L^{2(k+1)}}\Big\}< \infty, 
\end{equation}
then the application
\begin{equation}\label{eq0.9x3sm}
t\mapsto\tilde{\mathcal{L}}(f)(t):=\int_0^t V(t-t')(f_x)^{k+1} (t') dt',\qquad 0< |t|\le T\le 1,
\end{equation}
is continuous from  $[-T, T]\backslash\{0\}$ to  $H^{s+\mu}$. 
\end{proposition}
\begin{proof}
The proof follows by using \eqref{eq0.99} and a similar procedure applied in the proof of Proposition 2.9.
\end{proof}

\section{Proof of the main results}\label{sec-3}

Having the linear and nonlinear estimates at hand from the previous section, now we provide proof of the main results of this work.

\begin{proof}[Proof of Theorem \ref{teorp-1}]
Consider the Cauchy problem \eqref{eq:hs} in its equivalent integral form
\begin{equation}\label{int1sma0} 
v(t)=V(t)v_{0}- \int_{0}^{t}V(t-\tau)\partial_x(v^{k+1})(\tau)d\tau,
\end{equation}
where 
$V(t)$ is the semi-group associated with the linear part given by (\ref{gV}).

Let us define an application
\begin{equation}\label{int2sma0}
  \Psi(v)(t)= V(t)u_0-\int_0^t V(t-\tau)\partial_x(v^{k+1})(\tau)d\tau.
\end{equation}

 For $s \geq -1$, $r>0$ and $0<T\leq 1$, let 
\begin{align*}
  B_r^T= \{f\in \mathcal{X}_T^s ; \,\,\|f\|_{\mathcal{X}_T^s}\leq r \},
\end{align*}
be a ball in $\mathcal{X}_T^s$ with center at origin and radius $r$. We will show that there exists $r>0$ and $0<T\leq 1$ such that the application  $\Psi$ maps  $B_r^T $ into 
$B_r^T$ and is a contraction. For this, let $v\in B_r^T$.
Using the estimates \eqref{mm.2}  and  \eqref{mm.04}, one can obtain
\begin{equation}\label{eq3.5sma0}
\begin{split}
  \|\Psi(v)\|_{\mathcal{X}_T^s} \le C\|v_0\|_{H^s}+C_{p,k} T^{\omega_k}\| v\|_{\mathcal{X}_T^s}^{k+1},
\end{split}
\end{equation}
where $\omega_k=\frac{2p-3k-2}{2p}>0$. 

 For $v\in B_r^T$ let us choose  $r=4c\|v_0\|_{H^s}$ in such a way that $cT^{\omega_k} r^k=1/4$. For this choice, one can easily obtain
 \begin{equation}\label{eq3.8sma0}
  \|\Psi(v)\|_{X_T^s}\leq \frac{r}{4}+ cT^{\omega_k} r^{k+1}\leq  \frac{r}{2}.
\end{equation}

From \eqref{eq3.8sma0} we conclude that $\Psi$ maps $B_r^T$ into itself. A
similar argument proves that $\Psi$ is a contraction.  Hence $\Psi$ has a fixed point $v$ which is a  solution of the Cauchy problem (\ref{eq:hs}) such that 
$v \in C([0,T], H^s(\R))$. The rest of the proof follows from standard argument, see for example \cite{kpv1:kpv1}.

The regularity part follows using Lemma \ref{lem-smooth} and  Proposition \ref{xav7smoth} as in \cite{XC-MP3}.
\end{proof}

\begin{proof}[Proof of Theorem \ref{teorp}]
Proof of this theorem is very similar to that of Theorem \ref{teorp-1}. In this case, we use the estimate \eqref{linear1} from Lemma \ref{mm.1} and estimate \eqref{mm.05} from Lemma \ref{mm.03} to perform the contraction mapping argument.
\end{proof}

\bigskip
{\bf Acknowledgment.} The first author thanks supports from CNPq under grants
304036/2014-5 and 481715/2012-6.  The second author acknowledges supports from FAPESP  under grants 2012/20966-4, and CNPq under grants 479558/2013-2 and  305483/2014-5.   The authors would like to thank unanimous referees whose comments helped to improve the original manuscript.


\end{document}